\documentclass[12pt]{amsart}
\usepackage{amssymb,amscd,amsmath}
\usepackage[all]{xy}
\newtheorem{thm}{Theorem}

\newtheorem{cor}[thm]{Corollary}

\newtheorem{prop}[thm]{Proposition}

\newtheorem{prob}[thm]{Problem}
\newtheorem{lem}[thm]{Lemma}
\theoremstyle{remark}

\theoremstyle{definition}

\renewcommand{\bar}{\overline}

\newcommand{\Z}{{\mathbb{Z}}}
\newcommand{\uB}{{{B}}}
\newcommand{\uG}{{{G}}}
\newcommand{\uT}{{{T}}}
\newcommand{\cG}{{\mathcal{G}}}

\newcommand{\Aut}{\mathrm{Aut}\,}

\newcommand{\GL}{\mathrm{GL}}

\newcommand{\pr}{\mathrm{pr}}

\newcommand{\Spec}{\mathrm{Spec}\;}

\newcommand{\m}{\mathfrak m}

\newcommand{\Alt}{{\raise 2pt\hbox{$\scriptstyle\bigwedge$}}}

\newcommand{\e}{\epsilon}

\begin{document}
\title[A probabilistic Tits alternative and probabilistic identities]
{A probabilistic Tits alternative and probabilistic identities}

\author{Michael Larsen}
\email{mjlarsen@indiana.edu}
\address{Department of Mathematics\\
    Indiana University \\
    Bloomington, IN 47405\\
    U.S.A.}

\author{Aner Shalev}
\email{shalev@math.huji.ac.il}
\address{Einstein Institute of Mathematics\\
    Hebrew University \\
    Givat Ram, Jerusalem 91904\\
    Israel}

\thanks{ML was partially supported by NSF grant DMS-1401419.
AS was partially supported by ERC advanced grant 247034,
ISF grant 1117/13 and the Vinik Chair of mathematics which he holds.}

\begin{abstract}

We introduce the notion of a probabilistic identity of
a residually finite group $\Gamma $. By this we mean a non-trivial word
$w$ such that the probabilities that $w=1$ in the finite quotients 
of $\Gamma $ are bounded away from zero.

We prove that a finitely generated linear group satisfies
a probabilistic identity if and only if it is virtually solvable.

A main application of this result is a probabilistic variant of the Tits
alternative: let $\Gamma$ be a finitely generated linear group
over any field and let $G$ be its profinite completion. 
Then either $\Gamma$ is virtually solvable, or, for any $n \ge 1$,
$n$ random elements $g_1, \ldots , g_n$ of $G$ freely generate a
free (abstract) subgroup of $G$ with probability $1$.

We also discuss related problems and applications.
\end{abstract}

\maketitle

\newpage

\section{Introduction}

The celebrated Tits alternative \cite{T} asserts that a finitely generated
linear group is either virtually solvable or has a (nonabelian) free subgroup.
A number of variations and extensions of this result have been obtained
over the years.
In particular it is shown by Breuillard and Gelander \cite{BG}
that if $\Gamma$ is a finitely generated linear group which is
not virtually solvable then its profinite completion 
$\widehat{\Gamma}$ has a dense free subgroup of finite rank 
(this answers a question from \cite{DPPS}, where a somewhat 
weaker result was obtained). 
The purpose of this paper is to
establish a probabilistic  version of the Tits alternative, and to relate it
to the notion of probabilistic identities, which is interesting
in its own right.

In order to formulate our first result, 
let us say that a profinite group $G$ is {\em randomly free} if for any
positive integer $n$ the set of $n$-tuples in $G^n$ which
freely generate a free subgroup of $G$ (isomorphic to $F_n$)
has measure 1 (with respect to the normalized Haar measure
on $G^n$). We also say that a (discrete) residually finite
group $\Gamma$ is randomly free if its profinite completion
is randomly free.
 
Recall that related notions have already been studied in various contexts.
For example, in \cite{E} it is shown that connected finite-dimensional
nonsolvable real Lie groups are randomly free (in the sense that
the set of $n$-tuples which do not freely generate a free subgroup
has measure zero). We can now state our probabilistic Tits alternative.

\begin{thm} 
\label{first-cor}
Let $\Gamma$ be a finitely generated linear group over any field.
Then either $\Gamma$ is virtually solvable or $\Gamma$ is randomly free.
\end{thm}

The proof of this result relies on the notion and properties
of probabilistic identities which we introduce below.

Let $w=w(x_1, \ldots , x_n)$ be a non-trivial element of the free group $F_n$, 
and let $\Gamma $ be a residually finite group. 
Consider the induced word map $w_\Gamma : \Gamma ^n \rightarrow \Gamma $. 
If the
image (denoted by $w(\Gamma )$) of this map is $\{ 1 \}$ then $w$ is an
{\em identity} of $\Gamma $. We say that $w$ is a {\em probabilistic identity} of 
$\Gamma $ if there exists $\e > 0$ such that for each finite quotient $\Gamma/\Delta$
of $\Gamma $, the probability $P_{\Gamma/\Delta}(w)$ that $w(h_1, \ldots , h_n)=1$ 
(where the $h_i \in {\Gamma/\Delta}$ are chosen independently with respect to the uniform distribution 
on ${\Gamma/\Delta}$) is at least $\e$. This amounts to saying that, in the
profinite completion $\widehat{\Gamma }$ of $\Gamma $, the probability
(with respect to the Haar measure) that $w(g_1, \ldots , g_n) = 1$
is positive.

For example, $w=x_1^2$ is a probabilistic identity of the infinite
dihedral group $\Gamma  = D_{\infty}$, since in any finite quotient 
${\Gamma/\Delta} = D_n$ of $\Gamma $ we have $P_{\Gamma/\Delta}(w) \ge 1/2$.

More generally, probabilistic identities may be regarded as an extension
of the notion of coset identities.
Recall that a word $1 \ne w \in F_n$ is said to be a {\em coset identity} 
of the infinite group $\Gamma$ if there exists a finite index subgroup 
$\Delta\subseteq \Gamma$ and cosets $g_1\Delta, \ldots , g_n\Delta$ 
(where $g_i \in \Gamma$) such that 
$w(g_1\Delta, \ldots , g_n\Delta) = \{ 1 \}$.

Our main result on probabilistic identities is the following.

\begin{thm} 
\label{main}
Let $\Gamma $ be a finitely generated linear group over any field.
Then $\Gamma $ satisfies a probabilistic identity if and only if
$\Gamma $ is virtually solvable.
\end{thm}

Theorem~\ref{main} has several consequences. 

First, it implies that
a finitely generated linear group which satisfies a coset identity
is virtually solvable, a result which was proved in \cite{BG} 
(see Theorem 8.4 there) using other tools.

Secondly,  Theorem~\ref{main} easily implies Theorem~\ref{first-cor}. 
To show this, suppose $\Gamma$ is not virtually solvable, and let $G$ be 
the profinite completion of $\Gamma$. 
Note that $g_1, \ldots , g_n \in G$ freely generate a free
subgroup of $G$ if and only if $w(g_1, \ldots , g_n) \ne 1$
for every $1 \ne w \in F_n$. By Theorem~\ref{main} above, the probability
that $w(g_1, \ldots , g_n) = 1$ is $0$ for any such $w$.  As Haar measure is $\sigma$-additive,
the probability that there exists $w\neq 1$ such that $w(g_1, \ldots , g_n) = 1$
is also $0$.  Thus,
$g_1, \ldots , g_n$ freely generate a free subgroup with
probability $1$, proving Theorem~\ref{first-cor}.

Next, Theorem~\ref{main} immediately implies the following.

\begin{cor} A finitely generated linear group which satisfies
a probabilistic identity satisfies an identity.
\end{cor}

It would be interesting to find out whether the same holds
without the linearity assumption. We discuss this and related
problems and applications in Section 3 below.

Our original approach to proving Theorem~\ref{main} relied on strong 
approximation
for linear groups and on establishing upper bounds on the probabilities
$P_G(w)$ where $G$ is a group satisfying $T^k \le G \le \Aut(T^k)$
for a finite simple group $T$. However, this approach is rather
long and depends on the classification of finite simple groups.
A shorter and classification-free proof of Theorem~\ref{main} is given 
in Section 2 below.

The idea is to use linearity to map $\Gamma$ into a linear algebraic
group $G$ over an infinite product $\prod_{\m} A/\m$ of finite fields.
The closure of the image is then a profinite group.  Suppose that for some Zariski-closed subset $X\subset G^n$,  the measure of 
the closure of $X(\prod_{\m} A/\m)\cap \Gamma^n$
is positive.  Unless $X$ is a union of connected components of $G^n$, by 
intersecting $X$ with a suitable translate by an element of $\Gamma^n$, 
we find a smaller closed subset with the same property.
This process cannot continue indefinitely.  The theorem is obtained by applying this to any fiber of a non-trivial word map $w$.
The actual implementation uses the language of (affine) schemes.

In fact this method of proof yields the following extension of
Theorem 2: {\em Suppose $\Gamma$ is a finitely generated linear group
which is not virtually solvable. Then all fibers in $(\widehat{\Gamma})^n$ 
of all non-trivial words $w \in F_n$  have measure $0$.}

In other words, for a finite group $G$, let $P_{G,w}$ denote the
probability distribution induced on $G$ by $w$ (so that, for $g \in G$,
$P_{w,G}(g)$ is the probability that $w(g_1, \ldots , g_n) = g$).
Its $\ell_{\infty}$-norm is defined by 
$||P_{G,w}||_{\infty} = \max_{g \in G} P_{G,w}(g)$. Then we
have:

\begin{thm} Let $\Gamma$ be a finitely generated linear group.
Suppose for some $n \ge 1$ and $1 \ne w \in F_n$ there exists
$\epsilon > 0$ such that for all finite quotients $G$ of $\Gamma$ 
we have $||P_{G,w}||_{\infty} \ge \epsilon$. Then $\Gamma$
is virtually solvable.
\end{thm}.

\bigskip

\section{Proof of Theorem \ref{main}}

If a group $\Gamma$ acts on a space $X$ and $Y\subseteq X$, we say $Y$ is \emph{$\Gamma$-finite} if its orbit under $\Gamma$ is finite.
We say a closed subset $Z\subseteq X$ is \emph{$\Gamma$-covered} by  $Y$ if $Z$ is a closed subset
of some finite union of $\Gamma$-translates of $Y$.

\begin{lem}
\label{room}
Let $\Gamma$ be a group acting on a set $X$.  If $Y_1,\ldots, Y_n$ are subsets of $X$ which are not $\Gamma$-finite, then there exists $g\in \Gamma$ such that
$g Y_i \neq Y_j$ for $1\le i,j\le n$.
\end{lem}

\begin{proof}
For given $i,j$, the set of $g$ such that $g Y_i = Y_j$ is either empty or is a left coset of the stabilizer of $Y_i$ in $\Gamma$.  
By a theorem of B.\ H.\ Neumann \cite{N},
a group cannot be $\Gamma$-covered by a finite collection of left cosets of subgroups of infinite index.
\end{proof}

\begin{prop}
\label{measure}

Let $X$ be a Noetherian topological space and $\Gamma$ a group of homeomorphisms $X\to X$.
Let $f$ denote a function from the set of closed subsets of $X$ to $[0,1]$ satisfying the following  conditions:
\begin{enumerate}
\item[I.] If $Z\subseteq Y$ are closed subsets of $X$, then $f(Z)\le f(Y)$.
\item[II.] For all closed $Y\subseteq X$ and all $ g\in \Gamma$ such that $f(Y\cap  g Y)=0$, we have
$$f(Y\cup  g Y) \ge 2 f(Y).$$
\end{enumerate}
If $Y\subseteq X$ is closed and $\Gamma$-covers some closed subset $W\subseteq X$ with
$f(W) > 0$, then it $\Gamma$-covers some
closed $\Gamma$-stable subset $Z\subseteq X$ with $f(Z) > 0$.
 
 \end{prop}

\begin{proof}
By the Noetherian hypothesis, we may assume without loss of generality that $Y$ is minimal for the property of 
$\Gamma$-covering a set of positive $f$-value.
If two distinct irreducible components $Y_i$ and $Y_j$ of $Y$ were $\Gamma$-translates of one another, we could replace $Y$ with the union 
of all of its components except $Y_j$, and the resulting closed set would still $\Gamma$-cover a set of positive $f$-value.  This is impossible by the minimality of $Y$.

If $Y$ is $\Gamma$-finite, then 
$$Z := \bigcup_{g\in \Gamma} g Y$$
is a $\Gamma$-stable finite union of $\Gamma$-translates of $Y$ containing $W$.  By condition I, it satisfies $f(Z)>0$, so we are done. 
As $Y$ as a finite union of irreducible components, we may therefore assume at least one such component $Y_0$ is not $\Gamma$-finite.
We write  $Y = Y_0\cup Y'$, 
where no $\Gamma$-translate of $Y'$ contains $Y_0$.

By condition I,
there exists a finite sequence $g_1,\ldots,g_r\in \Gamma$ 
such that $f(Z) > 0$ for
$$Z := g_1 Y\cup \cdots \cup g_r Y.$$
We choose the $g_i$  so that
\begin{equation}
\label{halfplus}
f(Z) > \frac{\sup_{\Delta\subsetneq \Gamma\,\mathrm{ finite}}f\Bigl(\bigcup_{g\in \Delta} g Y \Bigr)}2.
\end{equation}

As no $\Gamma$-translate of $Y_0$ is $\Gamma$-finite, Lemma~\ref{room} implies there exists $g$ such that $g_i Y_0 \neq g g_j Y_0$ for all $i,j$.
Thus,
$$Y' \cup \bigcup_{i,j} (Y_0\cap g_i^{-1} g g_j Y_0)\subsetneq Y$$
$\Gamma$-covers $Z\cap g Z$.  By the minimality of $Y$, this means $f(Z\cap g Z) = 0$.
By condition II, $f(Z\cup gZ) \ge 2 f(Z)$, which contradicts (\ref{halfplus}).  We conclude that $Z$ must be $\Gamma$-finite.

\end{proof}

Now, let $A$ be an integral domain finitely generated over $\Z$ with fraction field $K$.  Let $\cG = \Spec B$ be an affine group scheme of finite type over $A$.
As usual, for every commutative $A$-algebra $T$, $\cG(T)$ will denote the set of $\Spec T$-points of $\cG$, i.e., the set of $A$-algebra homomorphisms $B\to T$.
The group structure on $\cG$ makes each $\cG(T)$ a group, functorially in $T$.  We regard $\cG$ as a topological space with respect to its Zariski topology.
If $Y\subseteq \cG$ is a closed subset, we  define
$Y(T)$ to be the subset of $\cG(T)$ consisting of $A$-homomorphisms $B\to T$ such that the corresponding map of topological spaces 
$\Spec T\to \cG$ sends $\Spec T$ into a subset of $Y$.  If $Z \subseteq \cG$ is another closed subset, then 
$$(Y\cap Z)(T) = Y(T) \cap Z(T),$$
but in general, the inclusion
$$Y(T) \cup Z(T)\subseteq (Y\cup Z)(T)$$
need not be an equality.

We define
$$P(\cG,A) := \prod_{\m \in \mathrm{Maxspec}(A)} \cG(A/\m),$$
where $\mathrm{Maxspec}$ denotes the set of maximal ideals, and $P(\cG,A)$ is
endowed with the product topology.  Note that as $\cG$ is of finite type (i.e.,  $B$ is a finitely generated $A$-algebra) and every $A/\m$ is a field finitely generated over $\Z$ (and hence finite), it follows that each $\cG(A/\m)$
is finite and $P(\cG,A)$ is a profinite group.  
For any closed subset $X\subseteq \cG$, we define the closed subset
$$P(X,A) := \prod_{\m \in \mathrm{Maxspec}(A)} X(A/\m)\subseteq P(\cG,A).$$

\begin{lem}
\label{non-generic}
If $X\subseteq \cG$ does not meet the generic fiber $\Spec B\otimes_A K\subset \cG$, then $P(X,A)$ is empty.
\end{lem}

\begin{proof}
If $I\subseteq B$ is the ideal defining $X$, then $(B/I)\otimes_A K = 0$, so $I\otimes_A K = B\otimes_A K$.  It follows that there exist elements $b_i\in I$ and $a_i/a'_i\in K$ such that
$$\sum_i b_i\otimes \frac{a_i}{a'_i} = 1,$$
and clearing denominators, we see that 
some non-zero element $f := \prod_i a'_i\in A$
belongs to $I$.  If $\m$ is a maximal ideal of $A[1/f]$, then $A[1/f]/\m$ is a field finitely generated over $\Z$, hence a finite field, 
and therefore $\m\cap A$ is a maximal ideal of $A$.  Thus, the image of $f$ in $A/(\m\cap A)$ is non-zero, from which it follows that
there are no $A$-homomorphisms $B/I\to A/(\m\cap A)$, i.e., $X(A/(\m\cap A))=\emptyset$.
\end{proof}

For any subgroup $\Gamma\subseteq \cG(A) \subseteq P(\cG,A)$, we define $\bar \Gamma$ to be the closure of $\Gamma$ in $P(\cG,A)$.  This is a closed subgroup of a profinite group and therefore a profinite group itself.  We endow it with Haar measure $\mu_{\bar \Gamma}$ normalized so that $(\bar\Gamma,\mu_{\bar \Gamma})$
is a probability space.  In particular, 
left translation by $\Gamma$ is a continuous measure-preserving action on
$(\bar\Gamma,\mu_{\bar \Gamma})$.  As Haar measure is outer regular, for every Borel set $B$,
$$\mu_{\bar \Gamma}(B)  = \inf_{S\subseteq \mathrm{Maxspec}(A)} \frac{|\pr_S B|}{|\pr_S \Gamma|},$$
where $S$ ranges over all finite sets of maximal ideals of $A$ and $\pr_S$ denotes projection onto $\prod_{\m\in S} \cG(A/\m)$.

For any positive integer $n$, we let $\cG^n$ denote the $n$th fiber power of $\cG$ relative to $A$, i.e., defining
$$B_n := \underbrace{B\otimes_A B\otimes_A\cdots\otimes_A B}_n,$$
we define $\cG^n := \Spec B_n$, regarded as a topological space with respect to the Zariski topology.  
Note that in general the Zariski topology on $\cG^n$ is \emph{not} the product topology.  However, by the universal property of tensor products, $\cG^n(T)$
is canonically isomorphic to $\cG(T)^n$ for all commutative $A$-algebras $T$.  Moreover, $B_n$ is a finitely generated $\Z$-algebra, and by the Hilbert basis theorem, this
implies that $\cG^n$ is a Noetherian topological space.

We consider the closure $\bar\Gamma^n$ of $\Gamma^n$ in $P(\cG^n,A)$.
For any closed subset $Y\subseteq \cG^n$, we define
$$P_\Gamma(Y) := \bar\Gamma^n\cap P(Y,A).$$
Thus, if $Y$ and $Z$ are closed subsets of $\cG^n$,
$$P_\Gamma(Y\cap Z) = \bar\Gamma^n\cap P(Y\cap Z, A) = \bar\Gamma^n\cap (P(Y,A)\cap P(Z,A)) 
= P_\Gamma(Y)\cap P_\Gamma(Z).$$
As
$$P(Y\cup Z,A) = \prod_{\m\in\mathrm{Maxspec}(A)} (Y(A/\m)\cup Z(A/\m)) \supseteq P(Y,A)\cup P(Z,A),$$
we have
$$P_\Gamma(Y\cup Z)\supseteq P_\Gamma(Y)\cup P_\Gamma(Z).$$
Defining
$$f(Y) := \mu_{\bar\Gamma^n}(P_\Gamma(Y)),$$
condition I  of Proposition~\ref{measure} is obvious.
As $\mu_{\bar\Gamma^n}$ is a measure, if $f(Y\cap Z) = 0$, then
\begin{align*}
f(Y\cup Z)& = \mu_{\bar\Gamma^n}(P_\Gamma(Y\cup Z))  \\
&\ge \mu_{\bar\Gamma^n}(P_\Gamma(Y)\cup P_\Gamma(Z))  \\
& =  \mu_{\bar\Gamma^n}(P_\Gamma(Y)) +  \mu_{\bar\Gamma^n}(P_\Gamma(Z)) -  \mu_{\bar\Gamma^n}(P_\Gamma(Y)\cap P_\Gamma(Z)) \\
&= f(Y) + f(Z) - f(Y\cap Z)= f(Y) + f(Z).
\end{align*}
As $\mu_{\bar\Gamma^n}$ is $\Gamma^n$-invariant, this implies condition II.  

\begin{prop}
\label{finite-orbit}
Let $G$ denote a linear algebraic group over a field $K$.
If $\Gamma$ is Zariski dense in $G(K)$, then a non-empty closed subset $Y$ of $G^n$  is $\Gamma^n$-finite
if and only if it is a union of connected components of $G^n$.
\end{prop}

\begin{proof}
If $Y$ is $\Gamma^n$-finite,  its stabilizer $\Delta$ is of finite index in $\Gamma^n$, which implies
that the Zariski closure $D$ of $\Delta$ in $G^n$ has finite index in $G^n$.  Thus $D\cap (G^n)^\circ$ is of finite index in $(G^n)^\circ$.
As $(G^n)^\circ$ is connected, it follows that $D$ contains $(G^n)^\circ$.  The Zariski-closure of any left coset of $\Gamma^n$ is a left coset of $D$
and therefore a union of cosets of $(G^n)^\circ$.  Conversely, any left translate of a coset of $(G^n)^\circ$ is again such a coset, so the orbit of any
union of connected components of $G^n$ is finite.
\end{proof}

\begin{prop}
\label{nonconstant}
Let $K$ be a field and $\uG$  a reductive algebraic group over $K$ with adjoint semisimple identity component.  
Let $w\in F_n$ be a non-trivial word and $g_0\in \uG(K)$.
Then $w^{-1}(g_0)$ does not contain any connected component of $\uG^n$.
 \end{prop}

\begin{proof}
Without loss of generality, we assume $K$ algebraically is closed.
Let $\uG^\circ$ be the identity component, $\uT$ a maximal torus of $\uG^\circ$ and $\uB$ a Borel subgroup of $\uG^\circ$
containing $\uT$.  Let $\Phi$ be the root system of $\uG$ with respect to $\uT$, and let $\Phi^+$ denote the set of roots of
$\uB$ with respect to $\uT$.
Every maximal torus of $\uG^\circ$ is conjugate under $\uG^\circ(K)$ to $\uT$.
The Weyl group $N_{\uG}(\uT)/\uT$ acts transitively on the set of Weyl chambers, so every pair $\uT'\subset \uB'$ is conjugate to $\uT\subset \uB$
by some element of $\uG^\circ(K)$.  In particular, for any $g\in \uG(K)$, the pair $g^{-1} \uT g\subset g^{-1}\uB g$ is conjugate in $\uG^\circ(K)$ to $\uT\subset \uB$,
or equivalently, there is some element $h\in g\uG^\circ(K)$ such that conjugation by $h$ stabilizes $\uT$ and $\uB$.
The highest root $\alpha$ of $\Phi^+$ is determined by $\uB$,
so $h$ likewise preserves $\alpha$.  It therefore normalizes $\ker\alpha^\circ$, and therefore the derived group $\uG_\alpha$ of
the centralizer of $\ker\alpha^\circ$.  This group is semisimple and of type $A_1$, so every element that normalizes it acts by an inner automorphism.  It follows that the centralizer of $\uG_\alpha$ in $\uG$ meets every connected component of $\uG$.

Suppose now that $w$ is constant on $g_1\uG^\circ\times\cdots\times g_n\uG^\circ$ for some $g_1,\ldots,g_n\in \uG(K)$.
Without loss of generality we may assume that all $g_i$ centralize $\uG_\alpha$.  As $w$ is constant on $g_1\uG_\alpha\times\cdots\times g_n\uG_\alpha$, and as
$$w(g_1h_1,\ldots,g_nh_n) = w(g_1,\ldots,g_n)w(h_1,\ldots,h_n)$$
for all  $h_1,\ldots,h_n\in \uG_\alpha(K)$, it follows that $w$ is constant on $\uG_\alpha^n$.  This is impossible because non-trivial words give non-trivial word maps on
all semisimple algebraic groups \cite{Borel}.
\end{proof}

We can now prove Theorem~\ref{main}.

\begin{proof}
It is clear that every virtually solvable group satisfies an identity and therefore satisfies a probabilistic identity.
For the other direction, we
fix a faithful representation $\Gamma\to \GL_k(F)$, where $F$ is a field, which we may assume algebraically closed.
Let $H\subset \GL_{k}$ denote the Zariski closure of $\Gamma$ in $\GL_k$.
Let $G$ denote the quotient of $H$ by its maximal solvable normal subgroup.  As $\Gamma$ is not virtually solvable, the same is true of $G$, so $H$ is not finite.  Its identity component is therefore a non-trivial adjoint semisimple group.

As $G$ is a linear algebraic group, we can regard it as a closed subgroup of $\GL_r$ for some $r$, so we have a homomorphism
$\rho\colon \Gamma\to \GL_r(F)$ such that the Zariski closure of $\rho(\Gamma)$ is $G$.
We recall how to extend $G$ to a subgroup scheme of $\GL_r$ defined over a finitely generated $\Z$-algebra. 
Let 
$$R_{\Z,r} := \Z[x_{ij},y] _{i,j=1,\ldots,r}/ (y \det(x_{ij}) - 1)$$
denote the coordinate ring of $\GL_r$ over $\Z$, and let 
$$\Delta_{\Z,r}\colon R_{\Z,r}\to R_{\Z,r}\otimes_{\Z} R_{\Z,r},$$
$$S_{\Z,r}\colon R_{\Z,r}\to R_{\Z,r},$$
and
$$\epsilon_{\Z,r}\colon R_{\Z,r}\to \Z$$
denote the ring homomorphisms associated to the
multiplication, inverse, and unit maps.  Closed subschemes of $\GL_r$ over any commutative ring  $A$ are in one-to-one correspondence with  ideals $I$ of $R_{A,r} := A\otimes_{\Z} R_{\Z,r}$, and
such an ideal defines a group subscheme if and only $I$ is a Hopf ideal \cite[\S2.1]{Waterhouse}, i.e., if and only if it satisfies the following three conditions:
$$\Delta_{A,r}(I)\subseteq I\otimes_A R_{A,r} + R_{A,r}\otimes_A I,$$
$$S_{A,r}(I) \subseteq I,$$
$$\epsilon_{A,r}(I) = \{0\}.$$

We fix a finite set of generators $h_k$ of the ideal $I_F$ in $R_{F,r}$ associated to $G$ as a closed subvariety
of $\GL_r$ over $F$.  
We lift each $h_k$ to an element $\tilde h_k\in F[x_{ij},y]$.  For any subring $A\subseteq F$ such that $\tilde h_k \in A [x_{ij},y]$, we denote again
by $h_k$  the image of $\tilde h_k$ in $R_{A,r}$; this should not cause confusion.
Let $A_0$ denote the subring of $F$ generated by all matrix entries in $\GL_r(F)$ of  the $\rho( g_j)$, as $g_j$ runs over some finite generating set of $\Gamma$,
together with all coefficients of the $\tilde h_k$.  Let $I_0$ denote the ideal generated by the elements $h_k$ in $R_{A_0,r}$, and let 
$K$ denote the fraction field of $A_0$.  As 
$$\Delta_{A_0,r}(I_0)\subseteq I_0\otimes_{A_0}R_{K,r}+R_{K,r}\otimes_{A_0} I_0$$
and
$$S_{A_0,r}(I_0)\subseteq I_0\otimes_{A_0}R_{K,r},$$
there exists $f\in A_0$ such that
$$\Delta_{A_0,r}(h_i)\in  I_0\otimes_{A_0}R_{A_0[1/f],r}+R_{A_0[1/f],r}\otimes_{A_0} I_0$$
and
$$S_{A_0,r}(h_i) \in I_0\otimes_{A_0} A_0[1/f]$$
for all $i$, and therefore, setting $A := A_0[1/f]$ and $I := I_0\otimes_{A_0} A$, we have
that $I$ is a Hopf ideal of $R_{A,r}$.
We set $\cG:= \Spec R_{A,r}/I$, the closed group subscheme of $\GL_r$ over $A$ defined by $h_k\in R_{A,r}$.
By construction, $\rho(\Gamma)$ is a Zariski-dense finitely generated subgroup of $\cG(A)$.

Let $Y := w^{-1}(g)\subseteq \cG^n$.  We define $f$ as above.  If $f(Y) > 0$, then $Y$ $\Gamma$-covers a set of positive $f$-value, so by
Proposition~\ref{measure}, $Y$ $\Gamma$-covers a closed $\Gamma$-stable subset $Z$
with $f(Z)>0$.  By Lemma~\ref{non-generic}, $ Z$ must meet the generic fiber $G^n$ of $\cG^n$, which implies that $Y$ must meet the generic fiber.
Proposition~\ref{finite-orbit} now implies that $Z \cap G^n$ contains a connected component of $G^n$, and it follows that $Y\cap G^n$ contains a connected
component.  However, this is impossible by Proposition~\ref{nonconstant}.  
Thus, $f(Y) = 0$.

Therefore, for every $\epsilon$, there exists a finite set $S$ of maximal ideals of $A$ such that
$$\frac{|\pr_S w^{-1}(g)|}{\pr_S \Gamma^n} < \epsilon.$$
Defining $\Delta$ to be the kernel of $\pr_S$, we complete the
proof of Theorem \ref{main} (as well as of Theorem 4).

\end{proof}

\bigskip

\section{Open problems}

In this section we discuss related open problems concerning
finite and residually finite groups. 

\begin{prob}  
\label{open1}
Do all finitely generated residually finite groups which satisfy
a probabilistic identity satisfy an identity?
\end{prob}

We also pose a related, finitary version of Problem~\ref{open1}.

\begin{prob}  
\label{open2}
Is it true that, for any word $1 \ne w \in F_n$, any positive
integer $d$ and any real number $\e > 0$, there exists a word
$1 \ne v \in F_m$ (for some $m$) such that, if $G$ is a finite 
$d$-generated group satisfying $P_G(w) \ge \epsilon$, then $v$
is an identity of $G$?
\end{prob}

Clearly, a positive answer to Problem~\ref{open2} implies a positive
answer to Problem~\ref{open1}.
Both seem to be very challenging questions, which might have
negative answers in general.
However, in some special cases they are solved affirmatively. 
For example, 
if $w = [x_1,x_2]$ or $w = x_1^2$, then it is known (see \cite{Ne} and
\cite{Ma}) that
for a finite group $G$, if $P_G(w) \ge \e > 0$, then $G$ is
bounded-by-abelian-by-bounded (in terms of $\e$). This implies
affirmative answers to Problems \ref{open1} and \ref{open2} for these particular
words $w$.

In general we cannot answer these problems for words of
the form $x_1^k$ ($k>2$).
However, for a prime $p$, a result of Khukhro \cite{Kh}
shows that, if $G$ is a finitely
generated pro-$p$ group satisfying a coset identity $x_1^p$
(namely there is a coset of an open subgroup consisting of
elements of order $p$ or $1$) then $G$ is virtually nilpotent
(and hence satisfies an identity). 

Another positive indication is the result showing that for
a (nonabelian) finite simple group $T$ and a non-trivial word $w$ we have
$P_T(w) \rightarrow 0$ as $|T| \rightarrow \infty$
(see \cite{DPPS} for this result, and also \cite{Fibers} for upper
bounds on $P_T(w)$ of the from $|T|^{-\alpha_w}$). 
This implies that a finite simple group $T$ satisfying
$P_T(w) \ge \epsilon > 0$ is of bounded size, hence it satisfies
an identity (depending on $w$ and $\epsilon$ only).

Affirmative answers to Problems~\ref{open1} and ~\ref{open2} 
would have far reaching applications. The argument proving 
Theorem~\ref{first-cor} above also proves the following.

\begin{prop} Assume Problem~\ref{open1} has a positive answer,
and let $\Gamma$ be a finitely generated residually finite group.
Then either $\Gamma$ satisfies an identity or $\Gamma$ is randomly free.

In particular, 

(i) If $\Gamma$ does not satisfy an identity then $\widehat{\Gamma}$
has a nonabelian free subgroup.

(ii) If $\widehat{\Gamma}$ has a nonabelian free subgroup then
almost all $n$-tuples in  $\widehat{\Gamma}$ freely generate
a free subgroup.
\end{prop}

The next application concerns residual properties of free groups. 
It is well known that the free group $F_n$ is residually-$p$.
But when is it residually $X$ for a collection $X$ of finite
$p$-groups? If this is the case then $F_n$ is also residually $Y$
where $Y$ is the subset of $X$ consisting of $n$-generated $p$-groups.
Thus we may replace $X$ by $Y$ and assume all $p$-groups in $X$
are $n$-generated. It is also clear that if $F_n$ ($n > 1$) is residually
$X$ then the groups in $X$ do not satisfy a common identity
(namely they generate the variety of all groups).

It turns out that, assuming an affirmative answer to Problem~\ref{open2},
these conditions are also sufficient.

\begin{prop} 
\label{res-X}
Assume Problem~\ref{open2} has a positive answer. 
Let $n \ge 2$ be an integer, $p$ a prime, and $X$ a set of $n$-generated
finite $p$-groups. Then the free group $F_n$ is residually 
$X$ if and only if the groups in $X$ do not satisfy a common identity.
\end{prop} 

To prove this, suppose the groups in $X$ do not satisfy a common identity.
To show that $F_n$ is residually $X$ we have to find, for each 
$1 \ne w=w(x_1, \ldots , x_n) \in F_n$,
a group $G \in X$ and an epimorphism $\phi:F_n \rightarrow G$,
such that $\phi(w) \ne 1$. This amounts to finding 
a group $G \in X$ and an $n$-tuple $g_1, \ldots , g_n \in G$ generating $G$ 
such that $w(g_1, \ldots , g_n) \ne 1$ (and then $\phi$ is defined by
sending $x_i$ to $g_i$). Suppose, given $w$, that there is
no $G \in X$ with such an $n$-tuple. Then, for every $G \in X$, and every
generating $n$-tuple $(g_1, \ldots , g_n) \in G^n$, we have 
$w(g_1, \ldots , g_n) = 1$. Now, the probability that
a random $n$-tuple in $G^n$ generates $G$ is the probability
that its image in $V^n$ spans $V$, where $V = G/\Phi(G)$ is the
Frattini quotient of $G$, regarded a vector space of dimension
$\le n$ over the field with $p$ elements. This probability is at least 
$\epsilon:=\prod_{i=1}^n (1-p^{-i}) > 0$. Thus $P_G(w) \ge \epsilon$
for all $G \in X$. By the affirmative
answer to Problem~\ref{open2}, all the groups $G \in X$ satisfy a common
identity $v \ne 1$ (which depends on $w$, $n$ and $p$).
This contradiction proves Proposition~\ref{res-X}.

This argument can be generalized to cases when $X$ consists
of finite groups $G$ with the property that $n$ random elements
of $G$ generate $G$ with probability bounded away from zero.
See \cite{JP} and the references therein for the description
of such groups and the related notion of positively finitely
generated profinite groups.

\end{document}